\documentclass[11pt]{amsart}

\usepackage{amsmath, amsfonts, amssymb,amsthm}
\usepackage{amstext}
\usepackage{mathrsfs}

\usepackage{float,epsf,subfigure}
\usepackage[all,cmtip]{xy}
\usepackage{hyperref}
\usepackage{algorithm}
\usepackage{algorithmic}
\usepackage{mathtools}
\usepackage{float}
\usepackage{bm}
\theoremstyle{plain}
\newtheorem{theorem}{Theorem}[section]

\newtheorem{cor}[theorem]{Corollary}

\newtheorem{def-thm}[theorem]{Definition-Theorem}
\newtheorem{lemma}[theorem]{Lemma}

\newtheorem*{corob}{Corollary B}

\newtheorem*{tha}{Theorem A}

\theoremstyle{definition}

\newtheorem{remark}[theorem]{Remark}

\def\min{\mathop{\mathrm{min}}}

\begin{document}
\title[On  Picard's  Problem via Nevanlinna Theory II]{On  Picard's  Problem via Nevanlinna Theory II}
\author[X.-J. Dong]
{Xianjing Dong}

\address{School of Mathematical Sciences \\ Qufu Normal University \\ Qufu, Jining, Shandong, 273165, P. R. China}
\email{xjdong05@126.com}


\date{}
\maketitle \thispagestyle{empty} \setcounter{page}{1}

\begin{abstract}   This  work  continues   the author's earlier   work  (2026, Studia Mathematica)  on     Picard's problem:     is  every    meromorphic function on a complete noncompact K\"ahler manifold with nonnegative Ricci curvature  necessarily   a  constant,  if   it avoids   $3$ distinct values? In that prior    work,   a positive     answer was  obtained   under a growth condition  for  non-parabolic manifolds.  In this paper,    we   give  a full    solution to the   non-parabolic case by removing this growth condition via  a global Green function approach. For the  parabolic case,    to overcome the obstacle arising from  the absence  of a positive global  Green function,  we  introduce   a   heat kernel approach to Nevanlinna theory.  Based on  it,   we develop a Carlson-Griffiths theory, which gives the first systematic  result in Nevanlinna theory for parabolic K\"ahler manifolds. As a direct  application, we confirm  the parabolic case of Picard's problem under a  weak  growth condition.
   \end{abstract}

\setlength\arraycolsep{2pt}
\medskip

\vskip\baselineskip

\section{Introduction}

The classical  Picard theorem is one of the most   fundamental and profound results in complex analysis, 
which claims  
  that any nonconstant meromorphic function on $\mathbb C$ can avoid  at most $2$ values.
Its high-dimensional generalization  to meromorphic functions and holomorphic mappings on complex manifolds has become a central  theme 
in several complex variables  over the past several decades, 
with  numerous  results established      
(see  \cite{at1, Ad, Dong1, nonpara, Fuj,  Gr1, Gold0, Gold, CC1, CC2, CC3, Kob, Pet, CC4}). 
 A long-standing   problem    concerns    the extensions of  Picard  theorem  
   to  complete   noncompact   K\"ahler manifolds with nonnegative Ricci curvature, which is the so-called  Picard's problem---a natural  problem   motivated by  
 S.-T. Yau's
  celebrated     1975   work \cite{Yau}   on  the Liouville  theorem.  
We formalize   this problem   as  follows:  

\noindent\textbf{Picard's Problem.}   \emph{Let $M$ be a complete noncompact K\"ahler manifold with  nonnegative Ricci curvature.    
 Is a meromorphic function  on $M$ necessarily  a  constant  if it avoids  $3$ distinct values$?$}

Picard's problem  lies at the intersection of  Nevanlinna theory  and  K\"ahler geometry, which has attracted extensive attention  and 
many  classical  results  have  been  obtained    under    various   conditions. 
 For instance, S. Kobayashi \cite{Kob}  confirmed     the problem for  $M$  on which      a  complex Lie group   acts  transitively;   
    Goldberg-Ishihara-Petridis \cite{Gold} gave   an affirmative answer to   the problem   for  
    a    holomorphic mapping  $f: M\to \mathbb P^1(\mathbb C)\setminus\{0, 1, \infty\}$ with   bounded dilatation, 
     provided that       $M$ is   locally flat; 
                                                    A. Atsuji \cite{at1}  (see also \cite{Dong3})   showed   that 
                                                        every  nonconstant slowly growing meromorphic function on $M$ can avoid at most $2$ values;  etc. 
                                                                                                                               Unfortunately,  all existing results in the  literature are subject to   essential restrictions, leaving 
                                                                 the general unconditional case still open.

 It is well-known that  the Nevanlinna theory,   established      by  R. Nevanlinna  \cite{Nev} in 1925,  presents    a  pivotal   advancement     of   the classical Picard  theorem.  
  As  a powerful  theoretical  tool,   
     its  generalization  to complex manifolds is of 
          substantial significance. 
            For more  details,    we refer   to the  important      papers   
       \cite{at0, at1,  ahlfors, Cart, gri, Dong1,  gri1, Ng, nochka, ru00,  Shi, Sa,   Stoll,  wu}  and the   monographs  
    \cite{Ko, No, ru, Shabat}. Recently, the  author  \cite{nonpara}  revisited  
  the Picard's  problem  via    Nevanlinna theory,  
  who generalized         
  the   Carlson-Griffiths theory \cite{gri, gri1}  
                    to    non-parabolic complete    
       K\"ahler   manifolds  with   nonnegative Ricci curvature by introducing a novel global Green function approach. 
          This  prior work \cite{nonpara}
    gives   
           a positive   answer  to the non-parabolic case  of Picard's problem  under a     certain    growth condition. 
 We           recall  the main result in \cite{nonpara} 
           to set the stage for this  work.

  Let $M$ be a complete noncompact K\"ahler manifold with Laplace-Beltrami operator $\Delta.$ Let $p(t,x,y)$ denote  the  heat kernel  of  $M$ associated to $\Delta$ (see Section \ref{sec22}). 
  We say that $M$ is \emph{non-parabolic} if  
  $$G(x,y)=2\int_0^\infty p(t,x, y)dt<\infty, \ \ \ \     x\not=y,$$
 and \emph{parabolic} otherwise.  
When $M$ is  non-parabolic,     $G(x,y)$  is   
  the    minimal positive  global Green function of $\Delta/2$ for $M.$ 
    For further   details, we   refer  to   \cite{Li-Tam, L-T-W, Va0, Va}.  
    Note that the notion of parabolicity used here differs from that introduced by  W. Stoll \cite{Stoll}. 
  In the sense of Stoll,  a complex manifold is  parabolic if  it admits   a parabolic exhaustion  function. For Riemann  surfaces, 
  however,  the two notions  coincide.

  Fix a reference point $o\in M.$  Let $V(r)$ denote        
  the Riemannian volume of the geodesic ball   centered at $o$ with radius $r$ in $M,$ 
   and let $\rho(x)$ denote    the Riemannian distance function of    $x$ from $o.$ 
  Under the assumptions  that $M$   is  non-parabolic and   has   nonnegative  Ricci  curvature,   
 the  Li-Yau's estimate  \cite{Li-Yau} gives   constants $A, B>0$ such that  
 $$A\int_{\rho(x)}^\infty\frac{tdt}{V(t)}\leq G(o,x)\leq B\int_{\rho(x)}^\infty\frac{tdt}{V(t)}$$
 holds for all $x\in M.$ 
 For $r>0,$   define    
  $$\Delta(r)=\left\{x\in M: \    G(o,x)>A\int_r^\infty\frac{tdt}{V(t)}\right\}$$
   and  
 $$g_r(o,x)=G(o,x)-A\int_r^\infty\frac{tdt}{V(t)}.$$
 So, $o\in\Delta(r)$ for all $r>0,$ and  $\{\Delta(r)\}_{r>0}$ is a family of precompact  domains which exhausts   $M,$  and $g_r(o,x)$ is  the Green function of $\Delta/2$ for $\Delta(r)$ with a pole at $o$ satisfying Dirichlet boundary condition.    
  
   Let  $N$ be a complex projective manifold,  over which one  puts a Hermitian positive line bundle $(L, h)$ with  Chern form $c_1(L, h)>0.$
 Choose    an effective   divisor  $D\in|L|,$ where  $|L|$ denotes   the complete linear system of $L.$  Moreover, 
 we   use  $K_N$ to denote  the canonical line bundle over $N.$ 
   Let  $f: M\to N$ be a  meromorphic mapping.    
By   \cite{nonpara} (or see Section \ref{sec32}),  in terms of    $g_r(o,x),$
          one   can define  the  Nevanlinna's functions 
        $T_f(r, L), m_f(r,D), N_f(r,D), \overline{N}_f(r,D),$ $T_f(r, K_N)$ and $T(r,\mathscr R)$ on $\Delta(r),$ where $\mathscr R$ denotes  the Ricci form of $M.$  
 
 Set
 \begin{equation*}\label{Hr}
H(r,\delta)=\frac{1}{r}\left(\frac{V(r)}{r}\right)^{1+\delta}\int_{r}^\infty\frac{tdt}{V(t)}.
 \end{equation*} 
 \ \ \ \  The main result of \cite{nonpara} is the following  Second Main Theorem.  

     \begin{tha}[Dong, \cite{nonpara}]\label{main11}  
 Let $M$ be a non-parabolic  complete  non-compact  K\"ahler manifold with non-negative Ricci curvature. 
 Let $N$ be a complex projective manifold of complex dimension not greater than that  of $M.$
 Let $D\in|L|$ be a reduced divisor of simple normal crossing type,  where $L$ is a positive line bundle over $X.$ Let $f:M\rightarrow N$ be a differentiably non-degenerate meromorphic mapping.  Then  for any $\delta>0,$ there exists a subset $E_\delta\subseteq(0, \infty)$ of finite Lebesgue measure such that 
$$T_f(r,L)+T_f(r, K_N)+T(r, \mathscr R)\leq \overline N_f(r,D)+O\left(\log^+T_f(r,L)+\log H(r,\delta)\right)$$
holds for all $r>0$ outside $E_\delta.$  
\end{tha}

 \begin{corob}[Picard-type Theorem]\label{dde2}    Let $M$ be a non-parabolic complete noncompact   K\"ahler manifold with non-negative Ricci curvature.  Then,  every  meromorphic  function  
on $M$  can avoid  at most $2$  distinct values  if $f$ satisfies the growth condition 
$$\lim_{r\to\infty}\frac{\log r}{T_f(r, \omega_{FS})}=0,$$
where $\omega_{FS}$ is the Fubini-Study form on $\mathbb P^1(\mathbb C).$
\end{corob}

            While   this  corollary  presents     a  Picard-type   theorem,  it   does not      settle        the non-parabolic case of Picard's problem at all, 
            nor does   it  provide  any   insight  into  the parabolic case of Picard's problem.  
                                                                                        In this  paper,   
 we     consider       both  the  non-parabolic and   parabolic   Picard's problem. Our investigation   attains  two core objectives. 
            Firstly, we refine the Carlson-Griffiths theory developed in \cite{nonpara} with an   optimal   error term $\delta\log r$ instead   of the  error term $\log H(r,\delta).$ 
                        So, this  refinement   enables us to provide        
                            a full  unconditional solution to the non-parabolic Picard's problem. 
Secondly,  in order  to overcome the obstacle caused by the nonexistence of a 
 positive global Green function  
for parabolic manifolds, we 
introduce a novel heat kernel 
approach to  Nevanlinna theory---a    geometric-analytic tool served   to  define  Nevanlinna's functions via the heat kernel in place of Green function. With that   entirely  new setting, 
we   develop     a Carlson-Griffiths theory    for  parabolic  manifolds and  give    a positive    answer 
       to the parabolic Picard's problem 
              under a  weak   growth condition.  
           Here,    the   global Green 
                    function approach and  heat kernel approach 
 used    in this paper  are  new   tools  introduced by the author   that are 
                                                                                    expected to facilitate further  progress in  Nevanlinna theory for complex manifolds.  

   We    state  the main results below. 
 Let   $M$ be  a   complete noncompact  K\"ahler  manifold with non-negative Ricci curvature. 
 Let $(N, \omega)$ be a compact K\"ahler manifold with Ricci form $\mathscr R.$   
A  divisor $D$ on $N$ is said to be \emph{cohomologous to} $\omega,$ if there exists  a function $u_D\geq 0$ on $N$ with  
 $\omega-[D]=2dd^c[u_D]$ in the sense of currents.  Let $f: M\to N$ be a meromorphic mapping.

\noindent\textbf{A.  Non-parabolic Carlson-Griffiths  Theory}

For a non-parabolic manifold $M,$  as detailed  in  Section \ref{sec32},  
 we  well    define the  Nevanlinna's functions 
 $T_f(r,\omega), m_f(r, D), N_f(r,D), \overline{N}_f(r, D), T_f(r, K_N)$ and   $T(r,\mathscr R)$ on $\Delta(r)$  defined   before.  
    Based on a global Green approach, we establish  the  following  Second Main Theorem. 

     \begin{theorem}[=Theorem \ref{main}]\label{main11}  
 Let $M$ be an $m$-dimensional  non-parabolic complete noncompact    K\"ahler manifold with nonnegative Ricci curvature. 
 Let $(N, \omega)$ be a compact  K\"ahler manifold of complex dimension $n\leq m.$
 Let $D_1,\cdots, D_q$ be effective divisors in general position on $N$ such that each $D_j$ is cohomologous to $\omega.$
  Let $f:M\rightarrow N$ be a differentiably non-degenerate meromorphic mapping.  Assume that $q\omega-{\rm Ric}(\omega^n)>0.$ Then  for any $\delta>0,$ there exists a subset $E_\delta\subseteq(0, \infty)$ of finite Lebesgue measure such that 
$$qT_f(r, \omega)+T_f(r, K_N)+T(r, \mathscr R)\leq \sum_{j=1}^q\overline N_f(r,D_j)+O\left(\log^+T_f(r,\omega)+ \delta\log r\right)$$
holds for all $r>0$ outside $E_\delta.$    \end{theorem}

Note  that Theorem \ref{main11}  extends   Theorem A. The \emph{simple defect}  of $f$ with respect to $D$  is defined   by
$$ \bar\delta_f(D)=1-\limsup_{r\rightarrow\infty}\frac{\overline{N}_f(r,D)}{T_f(r,\omega)}.$$
Set 
$$\left[\frac{c_1(K_N^*)}{\omega}\right]=\inf\left\{s\in\mathbb R: \ \eta< s\omega;  \  \   ^\exists\eta\in c_1(K^*_N)\right\}.$$
\begin{cor}[Defect Relation]\label{app1}  Assume the same conditions as in Theorem $\ref{main11}.$ Then 
$$\bar\delta_f(D)
\leq  \left[\frac{c_1(K_N^*)}{\omega}\right].$$
\end{cor}

In particular,  we equip  $N=\mathbb P^n(\mathbb C)$ with  Fubini-Study metric $\omega_{FS}.$ 
Taking  $D=H_1+\cdots+H_q,$ where $H_1,\cdots, H_q$ are  $q$ distinct hyperplanes in general position  in $\mathbb P^n(\mathbb C).$  
A basic  fact asserts  that 
$$\left[\frac{c_1\big(K_{\mathbb P^n(\mathbb C)}^*\big)}{\omega_{FS}}\right]=n+1,$$
which  gives   that 
$\bar\delta_f(H_1)+\cdots+\bar\delta_f(H_q)\leq n+1.$
 It is therefore:  
\begin{cor}[Picard-type Theorem]\label{dde1}    Let $M$ be a non-parabolic complete noncompact   K\"ahler manifold with nonnegative Ricci curvature.  Then, 
 every  meromorphic  function   on $M$  reduces  to a constant if it omits $3$  distinct values. 
\end{cor}
 
\noindent\textbf{B.  Parabolic  Carlson-Griffiths Theory}

When $M$ is a parabolic manifold, 
 the lack of  the minimal  positive global Green function renders the global Green function approach ineffective since    
   $\Delta(r)$ cannot be defined using  $G(o,x)$ $(\equiv\infty).$  
   To overcome  this difficulty,  we introduce  the heat kernel of $M$ to  construct $\Delta(r).$

Fix a reference point $o\in M.$  By   Li-Yau's estimate  \cite{Li-Yau}, for any $0<\epsilon<1,$ there exist two constants $c_1,c_2>0$ such that  $p(t,o,x)$ satisfies the following two-sided Gaussian-type estimate: 
\begin{equation*}\label{eett}
\frac{c_1}{2V(\sqrt{t})}e^{-\frac{\rho(x)^2}{4(1-\epsilon)t}}\leq p(t,o,x)\leq \frac{c_2}{2V(\sqrt{t})}e^{-\frac{\rho(x)^2}{4(1+\epsilon)t}},
\end{equation*}
where $V(\sqrt t)$ stands for   the Riemannian volume of the geodesic ball centered at $o$ with radius $\sqrt{t}$ in $M,$ and $\rho(x)$ stands for   the Riemannian distance from  $x$  to $o.$
In particular,  we have $\epsilon=0$ and $c_1=c_2=1/2$ for $M=\mathbb C.$  Put
$$G_r(o,x)=2\int_0^rp(t,o,x)dt.$$
For $r>0,$ define 
$$\Delta(r)=\left\{x\in M: \     G_r(o,x)>c_1\int_0^r\frac{1}{V(\sqrt{t})}e^{-\frac{r^2}{4(1-\epsilon)t}}dt\right\}.$$
  Note that  $o\in\Delta(r)$ for all $r>0,$ and  the family $\{\Delta(r)\}_{r>0}$ exhausts $M$ (see Section \ref{sec41}). 
We further  define  
$$h_r(o,x)=G_r(o,x)-c_1\int_0^r\frac{1}{V(\sqrt{t})}e^{-\frac{r^2}{4(1-\epsilon)t}}dt.$$
The function  $h_r(o,x)$  vanishes  identically  on  the boundary  $\partial\Delta(r),$ but  is 
 not  the Green function for $\Delta(r).$ 
  As detailed  in Section \ref{sec42},  by virtue of    $h_r(o,x),$       the 
            Nevanlinna's functions $T_f(r,\omega), m_f(r, D), N_f(r,D), \overline{N}_f(r, D), T_f(r,K_N)$ and $T(r,\mathscr R)$ can be  well-defined  on $\Delta(r).$      

Hereafter,  the notation  $``\|"$  indicates    an inequality  which   holds outside a possible  exceptional subset of $r\in(0,\infty)$ of  finite Lebesgue measure.  
 Based on a heat  kernel  approach, we establish  the  following  Second Main Theorem. 

     \begin{theorem}[=Theorem \ref{main220}]\label{ttt1}  
Let $M$ be an $m$-dimensional  parabolic complete noncompact    K\"ahler manifold with nonnegative Ricci curvature. 
 Let $(N, \omega)$ be a compact  K\"ahler manifold of complex dimension $n\leq m.$
 Let $D_1,\cdots, D_q$ be effective divisors in general position on $N$ such that each $D_j$ is cohomologous to $\omega.$
  Let $f:M\rightarrow N$ be a differentiably non-degenerate meromorphic mapping.  Assume that $q\omega-{\rm Ric}(\omega^n)>0.$ Then  for any $0<\epsilon<1,$ we have 
  $$qT_f(r, \omega)+T_f(r, K_N)+T(r, \mathscr R)\leq \sum_{j=1}^q\overline N_f(r,D_j)+O\big(r^m\log^+T_f(\theta r,\omega)\big)  \big\|$$
holds,  where $\theta=2(1+\epsilon)/(1-\epsilon).$  
\end{theorem}

 Letting $\epsilon\to0,$  we  conclude    a  defect relation:  
\begin{cor}[Defect Relation]\label{app1}  Assume the same conditions as in Theorem $\ref{ttt1}.$ If
$$\lim_{r\to\infty}\frac{r^m\log^+T_f(2r,\omega)}{T_f(r,\omega)}=0,$$
then 
$$\bar\delta_f(D)
\leq  \left[\frac{c_1(K_N^*)}{\omega}\right].$$
\end{cor}

We say that $f$ satisfies the  \emph{double growth property}  if 
$$T_f(2r, \omega)\leq O(T_f(r,\omega)).$$
Under a growth condition,  an affirmative  answer is given as follows.   

\begin{cor}[Picard-type Theorem]\label{}    Let $M$ be an $m$-dimensional parabolic complete noncompact   K\"ahler manifold with nonnegative Ricci curvature.  Then, 
 every  meromorphic  function  $f$ on $M$  can avoid  at most $2$  distinct values  if
 $$\lim_{r\to\infty}\frac{r^m\log^+T_f(2r,\omega_{FS})}{T_f(r,\omega_{FS})}=0.$$
 In particular, the conclusion holds  if $f$ satisfies the  double growth property with      
 $$T_f(r,\omega_{FS})\geq O(r^{m+\epsilon_0})$$
 for some   $\epsilon_0>0.$
\end{cor}

\noindent\textbf{C. Examples}

 From  Section \ref{sec22} later, 
  $M$ is non-parabolic if and only if 
\begin{equation*}
\int_1^\infty\frac{tdt}{V(t)}<\infty. 
\end{equation*}
Applying  this volume criterion,   we note  that $\mathbb C^m$ is  non-parabolic  for  $m\geq2,$ and   parabolic   for $m=1.$  We provide  some nontrivial  examples below. 

  \noindent{a)  \emph{Non-parabolic  K\"ahler manifolds with nonnegative Ricci curvature}}
    \begin{enumerate}
     \item[$\bullet$]  Equip $\mathbb C^m\setminus\{\textbf{0}\}$   with  the metric $ds^2=\|dz\|^2/\|z\|^2$ 
          ($z=(z_1,\cdots,z_m)$).  
     For $m\geq2,$     $\mathbb C^m\setminus\{\textbf{0}\}$ 
           is a non-parabolic   noncompact K\"ahler manifold with zero  Ricci curvature;
    \item[$\bullet$]   Note   that any  compact Riemannian  manifold has     finite  volume. 
     Let  $T^n$ denote  an $n$-dimensional  complex torus equipped with the  metric induced 
 from  the standard Euclidean metric on $\mathbb C^n.$
  When  $m\geq2,$  by endowing    the product metric, 
  one   can see   from 
    the volume criterion that 
       $\mathbb C^m\times T^n$ is a non-parabolic   noncompact K\"ahler manifold with zero  Ricci curvature, 
   and  $\mathbb C^m\times \mathbb P^n(\mathbb C)$ is a 
       non-parabolic   noncompact K\"ahler manifold with positive  Ricci curvature; 
             \item[$\bullet$]       More generally,   if  we adopt   the product metrics, then
                          both     $\mathbb C^m\times  N$  
                    and  $\mathbb C^m\setminus\{\textbf{0}\}\times N$
                                 are    
                                                  non-parabolic   noncompact K\"ahler manifolds with nonnegative   Ricci curvature  for    $m\geq2,$  
                here        
             $N$ is  any   compact K\"ahler manifold with nonnegative Ricci curvature. 
         \end{enumerate}
 
  \noindent{b) \emph{Parabolic  K\"ahler manifolds  with nonnegative Ricci curvature}}
    \begin{enumerate}
             \item[$\bullet$]  The cylinder $\mathbb C^*\cong \mathbb C/\mathbb Z$ equipped  with the metric $ds^2=|dz|^2/|z|^2,$   is a parabolic   noncompact K\"ahler manifold with zero  Ricci curvature; 
 \item[$\bullet$]   
 Under   the product metric,  the volume criterion gives    that 
   $\mathbb C^*\times T^n$  is a  parabolic   noncompact K\"ahler manifold with zero   Ricci curvature. Similarly, 
     $\mathbb C\times T^n\times \mathbb P^k(\mathbb C)$ is a parabolic   noncompact K\"ahler manifold  with positive   Ricci curvature under the product metric; 
   
   \item[$\bullet$]   More generally,    under the product metrics,  both  $\mathbb C\times  N$  and  $\mathbb C^*\times N$ 
      are    parabolic   noncompact K\"ahler manifolds with nonnegative   Ricci curvature,   here         
             $N$ is any   compact K\"ahler manifold with nonnegative Ricci curvature.  
                        \end{enumerate}

  \section{Preliminaries}

 \subsection{Ricci Curvature}~\label{sec21}
 
  Let $M$ be an  $n$-dimensional  Riemannian manifold, and $R$   the Riemannian curvature tensor of $M.$ 
 The sectional curvature of $M$ at a point $x\in M$ along a section $\Pi$ spanned by two nonzero   $X, Y\in T_{x}M$ is defined as  
 $$K(\Pi)=\frac{R(X, Y, Y, X)}{\|X\wedge Y\|^2},$$
 where 
 $$\|X\wedge Y\|^2=\|X\|^2\|Y\|^2-\langle X,Y\rangle^2.$$
Let $\{e_1,\cdots, e_{n}\}$ be an orthonormal basis of $T_{x}M.$
The Ricci curvature tensor of $M$ is defined by  
 $${\rm{Ric}}(X, Y)=\sum_{j=1}^{n} R(X, e_j, e_j, Y)$$
for  $X, Y\in T_{x}M,$ which is independent of the choice of  orthonormal basis of $T_{x}M.$ 
Then, the  Ricci curvature  of $M$ at $x$ in the direction $X$ is  defined  as 
 $${\rm{Ric}}(X)=\frac{{\rm{Ric}}(X, X)}{\|X\|^2}.$$ 
  We say that $M$ is of   non-negative Ricci curvature at $x,$ if ${\rm{Ric}}(X)\geq0$ for each nonzero vector  $X\in T_{x}M;$ and that  $M$ is of   non-negative Ricci curvature, if 
  it is of  non-negative Ricci curvature at each point  $x\in M.$ 
 
 \subsection{Volume Criterion for Non-parabolicity}~\label{sec22}
 
Let $M$ be a complete Riemannian manifold. 
We  introduce  a  criterion  for the  non-parabolicity  of $M$ using  volume growth. 
A  sharp necessary condition obtained by  N. Varopoulos \cite{Va1} states that if $M$ is non-parabolic, then 
\begin{equation}\label{cri}
\int_1^\infty\frac{tdt}{V(t)}<\infty,
\end{equation}
where $V(t)$ denotes    the Riemannian volume of a geodesic ball   centered at a  fixed  reference point  $o$ with radius $t$ in $M.$
However,   (\ref{cri}) is far from  sufficient,   and a counterexample was constructed  in \cite{Va1}.
The first major result for the sufficiency was due to N. Varopoulos \cite{Va} and Li-Yau \cite{Li-Yau}.
Based on  the   heat kernel estimate, they proved that if $M$ has non-negative Ricci curvature and (\ref{cri}) is satisfied, then $M$ is  non-parabolic.
So, for a Ricci nonnegatively curved manifold   $M,$  $M$ is non-parabolic if and only if (\ref{cri}) is satisfied.  In addition to the examples of such manifolds provided  in  Introduction, further 
 examples  can be found in \cite{Y-S, Y-T, T-Y}.

\subsection{Volume Comparison Theorem}~

A   space form  is  a  complete (simply-connected) Riemannian manifold with constant sectional curvature. 
Denote by $M^K$   the $n$-dimensional   space form   with constant sectional curvature $K.$ 
Let      $V(K, r)$ be  the Riemannian volume of a geodesic ball with radius $r$ in $M^K.$
   Let  $M$ be an    
$n$-dimensional complete Riemannian manifold with Ricci curvature  ${\rm{Ric}}_M,$ and   
     let  
  $V(r)$  stand for  the Riemannian 
  volume of  the geodesic ball centered at a fixed reference point $o$   with radius $r$ in $M.$

The well-known  volume comparison theorem by  Bishop-Gromov (see   \cite{B}) states  that  

\begin{theorem}\label{comp} If  ${\rm{Ric}}_M\geq(n-1)Kg,$  then the volume ratio $V(r)/V(K, r)$ is a non-increasing function in $r>0,$ and which  tends to $1$ as $r\to0.$  Hence, 
we have 
$$V(r)\leq V(K, r)$$
holds for  all $r>0.$
\end{theorem}

In particular,  when    $M^K=\mathbb R^n,$  we conclude  that    

\begin{cor}\label{volume}  If ${\rm{Ric}}_M\geq0,$ then we have 
$$V(r)\leq \omega_n r^n$$
holds for  all $r>0,$ where $\omega_{n}$ denotes  the volume of the unit ball in $\mathbb R^n.$
\end{cor}

When ${\rm{Ric}}_M\geq0,$   Calabi-Yau (see  \cite{S-Y})  proved that 

\begin{theorem}\label{volume1}  Assume   that $M$ is noncompact.  If  ${\rm{Ric}}_M\geq0,$  then $M$ has an infinite volume. More precisely, for any $\epsilon>0,$ there exists a constant $c=c(o, n, \epsilon)>0$ such that 
$$V(r)\geq c r$$
holds for  all $r\geq \epsilon.$
\end{theorem}

\subsection{Heat Kernel}~\label{sec22}

Let $M$ be a complete Riemannian manifold of dimension $n,$ with  Laplace-Beltrami operator $\Delta.$ 
We denote by  $\rho(x, y)$  the Riemannian distance between  $x, y\in M,$ and by $V_x(r)$ the Riemannian volume of the geodesic ball centered at $x$ with radius $r$ in $M.$ 
The  heat kernel $p(t, x, y)$ of  $M$  is  defined to be  the minimal positive fundamental solution of the  following heat equation 
$$\Big(\Delta-\frac{\partial}{\partial t}\Big)u(t, x)=0.$$
\ \ \ \  Li-Yau \cite{Li-Yau} (see \cite{S-Y} also) gave two-sided estimate    of  $p(t,x,y).$ 
\begin{theorem}\label{ppttt}  Assume that $M$ has non-negative Ricci curvature.  Then for any $0<\epsilon<1,$ there exist two constants $c_1=c_1(\epsilon, n),  c_2=c_2(\epsilon, n)>0$
such that 
$$\frac{c_1}{V_x(\sqrt t)}e^{-\frac{\rho(x, y)^2}{4(1-\epsilon)t}}\leq p(t, x, y)\leq \frac{c_2}{V_x(\sqrt t)}e^{-\frac{\rho(x, y)^2}{4(1+\epsilon)t}}$$
holds for all $x,y\in M$ and all $t>0.$
\end{theorem}

Let $\nabla$ denote the gradient operator on $M.$
\begin{theorem}\label{gradient} Assume that $M$ has nonnegative Ricci curvature.  Then for any $0<\epsilon<1,$ there exists a constant $c_2=c_2(\epsilon, n)>0$
such that 
$$\|\nabla_y p(t, x, y)\|\leq \frac{c_2}{\sqrt{t}V_x(\sqrt t)}e^{-\frac{\rho(x, y)^2}{4(1+\epsilon)t}}$$
holds for all $x\not=y\in M$ and all $t>0.$
\end{theorem}

 Define  
$$G(x, y)=2\int_0^\infty p(t, x, y)dt.$$
 This infinite  integral is  convergent if and only if $M$ is non-parabolic. When $M$ is non-parabolic, 
 $G(x, y)$ defines   the unique   minimal positive global Green function of $\Delta_y/2$ for $M,$ i.e., 
$$\begin{cases}
 -\frac{1}{2}\Delta_y G(x, y)=\delta_x(y),  \ \ & x, y\in M; \\
  G(x,y)>0,    &  x, y\in M;  \\
  \displaystyle{\lim_{\rho(x, y)\to\infty}}G(x,y)=0.
\end{cases}$$
where $\delta_x$ is the Dirac's delta  function with a pole at $x.$ 

When $M$ is non-parabolic with non-negative Ricci curvature,  Li-Yau \cite{Li-Yau}  (see \cite{S-Y} also) gave  two-sided estimate of $G(x,y)$ as follows.  
 \begin{theorem}\label{kkk}  Assume that $M$ is non-parabolic. If $M$ has   nonnegative Ricci curvature,   then   there exist two constants $A, B>0$ depending only on $n$ 
such that 
 $$A\int_{\rho(x, y)}^\infty\frac{tdt}{V_x(t)}\leq G(x,y)\leq B\int_{\rho(x, y)}^\infty\frac{tdt}{V_x(t)}$$
holds for all $x, y\in M.$
\end{theorem}

  \section{Carlson-Griffiths Theory for Non-parabolic Complete K\"ahler  Manifolds with Nonnegative Ricci Curvature}

Let  $(M, g)$ be   a   non-parabolic complete noncompact  K\"ahler manifold  with nonnegative Ricci curvature,
  of complex dimension $m.$ Let $\nabla$ be   the gradient operator on  $M.$
  The  K\"ahler form $\alpha$ of $M$ associated to    $g$ is defined as 
 $$\alpha=\frac{\sqrt{-1}}{\pi}\sum_{i,j=1}^mg_{i\bar j}dz_i\wedge d\bar z_{j}$$
in a local holomorphic coordinate system $z_1,\cdots,z_m.$
  
    \subsection{Construction of $\Delta(r)$}~\label{sec31}

 Fix a  reference point $o\in M.$ 
  Let   $V(r)$ be  the Riemannian volume of  $B(r),$ where  $B(r)$  denotes  the geodesic ball  centered at $o$ with radius  $r$ in $M.$   Note that  the non-parabolicity of $M$ implies that
 $$\int_1^\infty\frac{tdt}{V(t)}<\infty.$$
Then, we have the unique minimal positive global Green function $G(o,x)$ of the half Laplace-Beltrami operator $\Delta/2$  for $M$ with a pole at $o.$ In terms of heat kernel,   
we can write 
$$G(o, x)=2\int_0^\infty p(t, o,x)dt,$$
where $p(t,o,x)$ is the heat kernel of $M.$
  Let $\rho(x)$ be the Riemannian distance function of $x$ from $o.$ Using Theorem \ref{kkk}, there exist two constants $A, B>0$ such that  
 \begin{equation}\label{Gr}
 A\int_{\rho(x)}^\infty\frac{tdt}{V(t)}\leq G(o,x)\leq B\int_{\rho(x)}^\infty\frac{tdt}{V(t)}
 \end{equation}
holds for all $x\in M.$  For $r>0,$ define   
 $$\Delta(r)=\left\{x\in M: \    G(o,x)>A\int_r^\infty\frac{tdt}{V(t)}\right\}.$$
  By 
$$\lim_{\rho(x)\to 0}G(o,x)=\infty, \ \ \ \     \lim_{\rho(x)\to\infty}G(o,x)=0,$$
  one can conclude immediately    that  $\Delta(r)$ is a precompact  domain containing $o$  such  that 
 $$ \ \ \  \lim_{r\to0}\Delta(r)\to \emptyset, \ \ \ \   \lim_{r\to\infty}\Delta(r)=M.$$
Note that  the family 
     $\{\Delta(r)\}_{r>0}$ exhausts $M,$ i.e.,  for any   sequence $\{r_n\}_{n=1}^\infty$ with    $0<r_1<r_2<\cdots\to \infty,$ 
     we have       
 $$\bigcup_{n=1}^\infty\Delta(r_n)=M, \ \ \ \     \emptyset\not=\Delta(r_1)\subseteq\overline{\Delta(r_1)}\subseteq\Delta(r_2)\subseteq\overline{\Delta(r_2)}\subseteq\cdots$$    
The boundary $\partial\Delta(r)$ of $\Delta(r)$ can be formulated as
 $$\partial\Delta(r)=\left\{x\in M: \    G(o,x)=A\int_r^\infty\frac{tdt}{V(t)}\right\}.$$
   By  Sard's theorem,   $\partial\Delta(r)$  is a submanifold of $M$ for almost all $r>0.$  
    
     Set
 $$g_r(o,x)=G(o,x)-A\int_r^\infty\frac{tdt}{V(t)}.$$
Evidently,    $g_r(o,x)$ is  the positive Green function of $\Delta/2$ for $\Delta(r)$ with a pole at $o$ satisfying Dirichlet boundary condition, i.e., 
      $$\begin{cases}
 -\frac{1}{2}\Delta g_r(o,x)=\delta_o(x),  \ \ &    x\in\Delta(r); \\
   g_r(o,x)>0, \ \  &  x\in\Delta(r); \\
  g_r(o,x)=0, \ \  &  x\in\partial\Delta(r),
\end{cases}$$
 where $\delta_o$ is the Dirac's delta  function with a pole at $o.$ 
Let  $\pi_r$ stand for  the harmonic measure  on $\partial\Delta(r)$ with respect to $o,$ defined by
  $$d\pi_r=\frac{1}{2}\frac{\partial g_r(o,x)}{\partial{\vec{\nu}}}d\sigma_r,$$
  where  $\partial/\partial \vec\nu$ is the inward  normal derivative on $\partial \Delta(r),$ and $d\sigma_{r}$ is the induced Riemannian area element of 
$\partial \Delta(r).$

     \subsection{Nevanlinna's Functions and First Main Theorem}~\label{sec32}

   Below,  we   introduce   Nevanlinna's functions in the non-parabolic setting. 
  Let $(N, \omega)$ be a compact K\"ahler manifold, and $D$  an effective devisor on $N.$ We recall  that $D$ is cohomologous to $\omega,$ if  there exists a function 
$u_D\geq0$ on $N$ such that 
\begin{equation}\label{888}
\omega-[D]=2dd^c\left[u_D\right]
\end{equation}
in the sense of currents, where 
$$d=\partial+\overline{\partial}, \ \ \ \     d^c=\frac{\sqrt{-1}}{4\pi}\big(\overline{\partial}-\partial\big)$$
 so that      $$dd^c=\frac{\sqrt{-1}}{2\pi}\partial\overline{\partial}.$$
\ \ \ \     Let $f: M\to N$ be a meromoprhic  mapping. Define  
$$e_{f,\omega}=2m\frac{f^*\omega\wedge\alpha^{m-1}}{\alpha^m},$$
 which is called  the \emph{energy density function} of $f$ with respect to metrics $\alpha, \omega.$
 The \emph{characteristic function} of  $f$ with respect to $\omega$ is defined as
$$T_f(r, \omega)= \frac{1}{2}\int_{\Delta(r)}g_r(o,x) e_{f, \omega}dv,$$
where $dv$  stands for      the Riemannian volume element of $M.$ 
Assume that  $D$ is   cohomologous to $\omega.$ 
We define the  \emph{proximity function} of $f$ with respect to $D$ as  
$$m_f(r,D)=\int_{\partial\Delta(r)} u_D\circ fd\pi_r.$$
Finally, the 
\emph{counting function} and \emph{simple counting function} of $f$ with respect to $D$ are  respectively defined   as  
   \begin{eqnarray*}
N_f(r,D)&=& \frac{\pi^m}{(m-1)!}\int_{f^*D\cap\Delta(r)}g_r(o,x)\alpha^{m-1}, \\
\overline{N}_f(r, D)&=& \frac{\pi^m}{(m-1)!}\int_{f^{-1}(D)\cap\Delta(r)}g_r(o,x)\alpha^{m-1}.
 \end{eqnarray*}
 \ \ \ \     To establish the First Main Theorem,  we need  the  Green-Dynkin formula (see  \cite{at, Dong1, nonpara, DY}) as follows. 
\begin{lemma}[Green-Dynkin Formula]\label{dynkin} Let $\phi$ be a $\mathscr C^2$ function on  $M$ outside  a polar set of singularities  at most. Assume that $\phi(o)\not=\infty.$  Then
$$\int_{\partial \Delta(r)}\phi d\pi_{r}-\phi(o)=\frac{1}{2}\int_{\Delta(r)}g_r(o,x)\Delta \phi dv.$$
\end{lemma}
\begin{proof} 
Refer to  \cite{DY} for a probabilistic proof. Here, we provide  an  alternative  analytic proof. 
Since the possible  set of  singularities of $\phi$ is polar,  the    Green's second identity is applicable   to $\phi$ in the sense of distributions.  Thus, we have 
  \begin{eqnarray*}
&&\int_{\Delta(r)}g_r(o,x)\Delta \phi dv-\int_{\Delta(r)}\phi\Delta g_r(o,x) dv \\
&=& \int_{\partial\Delta(r)}\phi \frac{\partial g_r(o,x)}{\partial{\vec{\nu}}}d\sigma_r-\int_{\partial\Delta(r)}g_r(o,x) \frac{\partial\phi}{\partial{\vec{\nu}}}d\sigma_r,
  \end{eqnarray*}
  where  $\partial/\partial \vec\nu$ is the inward  normal derivative on $\partial \Delta(r),$ and $d\sigma_{r}$ is the induced Riemannian area element of 
$\partial \Delta(r).$  From the properties   of Green functions, we get 
$$\int_{\Delta(r)}\phi\Delta g_r(o,x) dv=-2\phi(o), \ \ \ \      \int_{\partial\Delta(r)}g_r(o,x) \frac{\partial\phi}{\partial{\vec{\nu}}}d\sigma_r=0.$$
Using the definition of harmonic measures, we also have 
$$ \int_{\partial\Delta(r)}\phi \frac{\partial g_r(o,x)}{\partial{\vec{\nu}}}d\sigma_r=2\int_{\partial\Delta(r)}\phi d\pi_r.$$
Combining  the above, we have the lemma proved.
\end{proof}

\begin{theorem}[First Main Theorem]\label{first}  Assume   that $f(o)\not\in{\rm{Supp}}D.$ Then
$$T_f(r, \omega)+u_D\circ f(o)=m_f(r,D)+N_f(r,D).$$
\end{theorem}
 \begin{proof} 
Let $\mathscr O(D)$ be the holomorphic line bundle over $N$ defined by $D.$
Let $s_D$ be the  canonical section  associated to $D,$ i.e., $s_D$ is a holomorphic section  of  $\mathscr O(D)$ over $N$ with zero divisor $D.$
Locally, we write  $s_D=\tilde s_D e,$ where $e$ is a local holomorphic frame of $\mathscr O(D).$ By Poincar\'e-Lelong formula, we obtain   
$[D]=dd^c[\log|\tilde s_D|^2]>0$
in the sense of currents. Moreover,   (\ref{888})  gives 
$$\omega=dd^c\big[2u_D+\log|\tilde s_D|^2\big]>0,$$
which leads to  that $2u_D+\log|\tilde s_D|^2$ is  plurisubharmonic.
Since $N$ is a K\"ahler manifold, we deduce   that  both $\log|\tilde s_D|^2$ and $2u_D+\log|\tilde s_D|^2$ are  subharmonic. 
Thus,  $u_D$ is the difference of two subharmonic functions. It implies that the set of singularities  of $u_D\circ f$ is polar.
Apply Green-Dynkin formula to $u_D\circ f,$  we obtain   
$$\int_{\partial\Delta(r)}u_D\circ fd\pi_{r}-u_D\circ f(o)=\frac{1}{2}\int_{\Delta(r)}g_{r}(o,x)\Delta (u_D\circ f)dv.$$
The first term on the left hand side equals $m_f(r, D),$ and the right hand side is  equal to  
 \begin{eqnarray*}
 &&  \frac{2\pi^m}{(m-1)!}\int_{\Delta(r)}g_{r}(o,x)dd^c[u_D\circ f]\wedge\alpha^{m-1}  \\
  &=&\frac{\pi^m}{(m-1)!}\int_{\Delta(r)}g_{r}(o,x)f^*\omega\wedge\alpha^{m-1} \\
  && -\frac{\pi^m}{(m-1)!}\int_{\Delta(r)}g_{r}(o,x)dd^c\left[\log|\tilde s_D\circ f|^2\right]\wedge\alpha^{m-1} \\
 &=& T_f(r, \omega)-N_f(r, D).
    \end{eqnarray*}
    This completes the proof.
\end{proof}

 \subsection{Calculus Lemma}~

  To  compute the gradient  of $g_r(o,x),$  we need  the following    lemma. 
   \begin{lemma}\label{thm4}    For  $0<t\leq r,$ we have
 $$g_r(o,x)=A\int_t^r \frac{sds}{V(s)}$$
 on $\partial\Delta(t),$  where $A$ is given by $(\ref{Gr}).$
  \end{lemma}
 \begin{proof}    Using  the definition of $g_r(o,x),$   we get  
  \begin{eqnarray*} 
g_r(o,x)&=& G(o,x)-A\int_{r}^\infty\frac{tdt}{V(t)}  \\
&=& G(o,x)-A\int_{t}^\infty\frac{sds}{V(s)}ds +A\int_{t}^r\frac{sds}{V(s)}  \\
 &=& g_t(o,x)+A\int_{t}^r\frac{sds}{V(s)}.   
   \end{eqnarray*}
 Since $g_t(o,x)=0$ on $\partial\Delta(t),$ 
  we have the desired result.
 \end{proof}

 \begin{theorem}\label{hh}  For  $0<t\leq r,$ we have
 $$\|\nabla g_r(o,x)\|=\frac{At}{V(t)}$$
  on $\partial\Delta(t).$
\end{theorem}
\begin{proof}  By Lemma \ref{thm4}, $g_r(o,x)$ is  a constant depending only on the parameter  $t$ on $\partial\Delta(t)$ with  $0<t\leq r,$ which means  that the gradient direction of $g_r(o,x)$ at a point 
 $x\in \partial\Delta(t)$ is 
exactly the direction in which  $t$ decreases.  Thus, one obtains
 $$\|\nabla g_r(o,x)\|=-A\frac{d}{dt}\int_t^r \frac{sds}{V(s)}=\frac{At}{V(t)}$$
 on $\partial\Delta(t).$
  \end{proof}
   
 Since $\partial\Delta(t)$ is  a lever surface with parameter $t$ determined  by $g_r(o,x)$ for  $0<t\leq r,$ the gradient direction of $g_r(o,x)$ at  a point $x\in \partial\Delta(t)$ is also   the inward normal 
 direction of $\partial\Delta(t)$ at $x.$ By Theorem \ref{hh}, we have 
    $$\frac{\partial g_r(o,x)}{\partial{\vec{\nu}}}=\|\nabla g_r(o, x)\|=\frac{Ar}{V(r)}$$
on $\partial\Delta(r).$ From the definition of   $d\pi_r,$      it is immediate that 

  \begin{cor}\label{bbb}  We have 
   $$d\pi_r= \frac{A}{2}\frac{r}{V(r)}d\sigma_r,$$
 where  $d\sigma_{r}$ is the induced Riemannian area element of 
$\partial \Delta(r).$
\end{cor}

  We need the following Borel's growth  lemma (see  \cite{nonpara, No, ru}). 
 \begin{lemma}[Borel's Growth Lemma, \cite{nonpara}]\label{} Let $u\geq0$ be a non-decreasing  function  on $(r_0, \infty)$ with $r_0\geq0.$ Then for any $\delta>0,$ there exists a subset $E_\delta\subseteq(r_0,\infty)$
 of finite Lebesgue measure such that  
 $$u'(r)\leq u(r)^{1+\delta}$$
 holds for all $r>r_0$ outside $E_{\delta}.$  
 \end{lemma}

We establish  the following  Calculus Lemma. 

 \begin{theorem}[Calculus Lemma]\label{calculus}
Let $k\geq0$ be a locally integrable function on $M.$ Assume that $k$ is locally bounded at $o.$ Then for any $\delta>0,$  there exists  a subset $E_{\delta}\subseteq(0,\infty)$ of finite Lebesgue measure such that
$$\int_{\partial\Delta(r)}kd\pi_r\leq \frac{1}{2}\left(\frac{V(r)}{r}\right)^{\delta}\bigg(\int_{\Delta(r)}g_r(o,x)kdv\bigg)^{(1+\delta)^2}$$
holds for all $r>0$ outside $E_{\delta}.$  
\end{theorem}
 
  \begin{proof} 
 Invoking  Lemma  \ref{thm4}, we have   
   \begin{eqnarray*}
 \int_{\Delta(r)}g_r(o,x)kdv 
 &=&\int_0^rdt \int_{\partial\Delta(t)}g_r(o,x)kd\sigma_t \\
    &=&  A\int_0^r\left(\int_{t}^r\frac{sds}{V(s)}\right)dt \int_{\partial\Delta(t)}kd\sigma_t.
   \end{eqnarray*}
 Set 
 $$\Lambda(r)=A\int_0^{r}\left(\int_{t}^{r}\frac{sds}{V(s)}\right)dt \int_{\partial\Delta(t)}kd\sigma_t.$$
 A simple computation leads to  
 $$\Lambda'(r)=\frac{Ar}{V(r)}\int_0^rdt\int_{\partial\Delta(t)}kd\sigma_t.$$
 In further, we have
 $$\frac{d}{dr}\left(\frac{V(r)\Lambda'(r)}{r}\right)=A\int_{\partial\Delta(r)}kd\sigma_r.$$
Employing   Borel's growth lemma  to the left hand side of the above equality  twice:  one is to $V(r)\Lambda'(r)/r$ and  another   is to 
 $\Lambda'(r),$ then  we conclude   that for any  $\delta>0,$   there exists a subset $E_\delta\subseteq(0,\infty)$ of finite Lebesgue measure such that 
$$ \int_{\partial\Delta(r)}kd\sigma_r\leq \frac{1}{A}\left(\frac{V(r)}{r}\right)^{1+\delta}\Lambda(r)^{(1+\delta)^2}
$$ holds for all $r>0$ outside $E_\delta.$  
  On the other hand,  Corollary \ref{bbb} gives  
   $$d\pi_r= \frac{A}{2}\frac{r}{V(r)}d\sigma_r.$$
Put together  the above, we conclude that  
   \begin{eqnarray*}
 \int_{\partial\Delta(r)}kd\pi_r &\leq& \frac{1}{2}\left(\frac{V(r)}{r}\right)^{\delta}\Lambda(r)^{(1+\delta)^2}\\ 
 &=&\frac{1}{2}\left(\frac{V(r)}{r}\right)^{\delta}\bigg(\int_{\Delta(r)}g_r(o,x)kdv\bigg)^{(1+\delta)^2}.
   \end{eqnarray*}
 holds for all $r>0$ outside $E_\delta.$  
 \end{proof}

 \begin{cor}\label{calculus1}
Let $k\geq0$ be a locally integrable function on $M.$ Assume that $k$ is locally bounded at $o.$ Then for any $\delta>0,$   there exists   a subset $E_{\delta}\subseteq(0,\infty)$ of finite Lebesgue measure such that
$$\log^+\int_{\partial\Delta(r)}kd\pi_r\leq  (1+\delta)^2\log^+\int_{\Delta(r)}g_r(o,x)kdv+O\big(\delta\log r\big)$$
holds for all $r>0$ outside $E_{\delta}.$  
\end{cor}

  \subsection{Second Main Theorem}~
  
      Define    
$$T(r,\mathscr R)= \frac{\pi^m}{(m-1)!}\int_{\Delta(r)}g_r(o,x)\mathscr R\wedge\alpha^{m-1},$$
where $\mathscr R=-dd^c\log\det(g_{i\bar j})$ is the Ricci form of  $M.$  

Let $D_1,\cdots, D_q$ be  effective divisors  on $N$ such that each $D_j$ is cohomologous to $\omega,$ i.e., there exist   functions $u_{D_1},\cdots, u_{D_q}\geq0$ on $N$ such that 
$$\omega-[D_j]=2dd^c[u_{D_j}], \ \ \ \  j=1,\cdots,q.$$ 
Set
$$\Psi=\frac{\omega^n}{\prod_{j=1}^q u_{D_j}^{2}e^{-2u_{D_j}}}.$$

\begin{lemma}\cite{gri}\label{grii}
Assume that   $D_1, \cdots, D_q$ are  in general position. 
If $qw-{\rm Ric}(\omega^n)>0,$ then 

$(a)$  $-{\rm Ric}\Psi\geq0;$

$(b)$ there exists a constant $c>0$ such that $(-{\rm Ric}\Psi)^n\geq c\Psi;$

$(c)$ $\int_{N\setminus D}(-{\rm Ric}\Psi)^n<\infty,$ where $D=D_1+\cdots+D_q.$
\end{lemma}

\begin{theorem}[Second Main Theorem]\label{main}   Let $M$ be a non-parabolic complete noncompact    K\"ahler manifold with nonnegative Ricci curvature. 
 Let $(N, \omega)$ be a compact  K\"ahler manifold of complex dimension not greater than that  of $M.$
 Let $D_1,\cdots, D_q$ be effective divisors in general position on $N$ such that each $D_j$ is cohomologous to $\omega.$
  Let $f:M\rightarrow N$ be a differentiably non-degenerate meromorphic mapping.  Assume that $q\omega-{\rm Ric}(\omega^n)>0.$ Then  for any $\delta>0,$ there exists a subset $E_\delta\subseteq(0, \infty)$ of finite Lebesgue measure such that 
$$qT_f(r, \omega)+T_f(r, K_N)+T(r, \mathscr R)\leq \sum_{j=1}^q\overline N_f(r,D_j)+O\left(\log^+T_f(r,\omega)+ \delta\log r\right)$$
holds for all $r>0$ outside $E_\delta.$    
\end{theorem}

\begin{proof}   
Set   $D=D_1+\cdots+D_q$ and define a function $\xi\geq0$ by
$$f^*\Psi\wedge\alpha^{m-n}=\xi\alpha^m.$$
By  Poincar\'e-Lelong formula and  (\ref{888}), we obtain   
 \begin{eqnarray*}
&&  dd^c\left[\log\xi\right] \\
&=& -f^*{\rm Ric}(\omega^n)+\mathscr R+[D_{f,\rm ram}]+2\sum_{j=1}^q f^*dd^c[u_{D_j}] \\
&& -2\sum_{j=1}^q f^*dd^c[\log u_{D_j}] \\
&=& qf^*\omega-f^*{\rm Ric}(\omega^n)+\mathscr R-f^*[D]+[D_{f,\rm ram}] 
 -2\sum_{j=1}^q f^*dd^c[\log u_{D_j}]
 \end{eqnarray*}
 in the sense of currents. 
In further,  using  Lemma \ref{dynkin} to  get 
 \begin{eqnarray*}
&& \frac{1}{2}\int_{\partial \Delta(r)}\log \xi d\pi_r \\
&=&\frac{1}{4}\int_{\Delta(r)}g_r(o,x)\Delta\log\xi dv +O(1) \nonumber \\
 &\geq& qT_f(r,\omega)+T_f(r, K_N)+T(r,\mathscr R)-N_f(r, D)+N(r, D_{f, \rm ram})   \nonumber \\
 && -\frac{1}{2}\sum_{j=1}^q\int_{\Delta(r)}g_r(o,x)\Delta \log(u_{D_j}\circ f) dv+O(1)  \nonumber \\
&\geq& qT_f(r,\omega)+T_f(r, K_N)+T(r,\mathscr R)-\overline{N}_f(r, D) \nonumber \\
&& -\sum_{j=1}^q\log \int_{\partial \Delta(r)}u_{D_j}\circ f d\pi_r +O(1) \nonumber \\
&=& qT_f(r,\omega)+T_f(r, K_N)+T(r,\mathscr R)-\overline{N}_f(r, D)  \\
&& -\sum_{j=1}^q\log m_f(r,D_j) +O(1).
 \end{eqnarray*}
Hence,  we conclude  from the  First Main Theorem that 
 \begin{eqnarray}\label{gjj}
 \frac{1}{2}\int_{\partial \Delta(r)}\log \xi d\pi_r  
&\geq& qT_f(r,\omega)+T_f(r, K_N)+T(r,\mathscr R)-\overline{N}_f(r, D)   \\
&& -q\log^+T_f(r,\omega)+O(1).  \nonumber 
 \end{eqnarray}
 Write
$$-f^*{\rm Ric}\Psi=\frac{\sqrt{-1}}{\pi}\sum_{i,j=1}^m\psi_{i\bar j}dz_i\wedge d\bar z_j.$$
For any $x_0\in M,$ one can take a   normal local holomorphic coordinate system $z_1,\cdots, z_m$  around  $x_0$ such that  
$$\alpha|_{x_0}=\frac{\sqrt{-1}}{\pi}\sum_{j=1}^m dz_j\wedge d\bar z_j.$$
By Lemma \ref{grii}
 \begin{eqnarray*}
f^*\Psi\wedge\alpha^{m-n}\big|_{x_0}&\leq& c^{-1}\left(-f^*{\rm Ric\Psi}\right)^n\wedge\alpha^{m-n} \\
&=& \bigg(\frac{\sqrt{-1}}{\pi}\sum_{i,j=1}^m\psi_{i\bar j}dz_i\wedge d\bar z_j\bigg)^{n}\bigwedge \bigg( \frac{\sqrt{-1}}{\pi}\sum_{j=1}^m dz_j\wedge d\bar z_j\bigg)^{m-n} \\
&=& c^{-1}(m-n)!\sum_{1\leq i_1\not=\cdots\not=i_n\leq m}\psi_{i_1\bar i_1}\cdots\psi_{i_n\bar i_n}\alpha^m \\
&\leq& c^{-1}(m-n)! \left({\rm tr}(\psi_{i\bar j})\right)^n\alpha^m.
 \end{eqnarray*}
On the other hand, we have 
\begin{equation*}
-f^*{\rm Ric\Psi}\wedge\alpha^{m-1}\big|_{x_0}=(m-1)!{\rm tr}(\psi_{i\bar j})\alpha^{m},
\end{equation*}
which  leads to 
 \begin{eqnarray*}
\xi^{\frac{1}{n}}\big|_{x_0}&=&\left(\frac{f^*\Psi\wedge\alpha^{m-n}}{\alpha^m}\right)^{\frac{1}{n}} \\
&\leq& c^{-\frac{1}{n}}(m-n)!^{\frac{1}{n}} {\rm tr}(\psi_{i\bar j}) \\
&=&-\frac{c^{-\frac{1}{n}}(m-n)!^{\frac{1}{n}}}{(m-1)!}\frac{f^*{\rm Ric\Psi}\wedge\alpha^{m-1}}{\alpha^m}.
 \end{eqnarray*}
Using  the arbitrariness  of $x_0,$ we obtain  
\begin{equation}\label{oppo}
\xi^{\frac{1}{n}}\leq -\frac{c^{-\frac{1}{n}}(m-n)!^{\frac{1}{n}}}{(m-1)!}\frac{f^*{\rm Ric\Psi}\wedge\alpha^{m-1}}{\alpha^m}.
\end{equation}
Since $\pi_r$ is a probability measure,   applying   Jensen's inequality and Corollary \ref{calculus1}    to get  
  \begin{eqnarray*}
 \int_{\partial \Delta(r)}\log \xi d\pi_r 
&\leq&  n\log\int_{\partial \Delta(r)}\xi^{\frac{1}{n}} d\pi_r \\
&\leq& n(1+\delta)^2\log \int_{\Delta(r)}g_r(o,x)\xi^{\frac{1}{n}}dv+O\big(\delta\log r\big) \\
&\leq& n(1+\delta)^2\log \int_{\Delta(r)}g_r(o,x)\frac{-f^*{\rm Ric\Psi}\wedge\alpha^{m-1}}{\alpha^m}dv+O\big(\delta\log r\big) \\
&=& n(1+\delta)^2\log T_f(r, -{\rm Ric}\Psi)+O\big(\delta\log r\big).
 \end{eqnarray*}
Since $N$ is compact and  $\omega>0,$ we have  
$$T_f(r, -{\rm Ric}\Psi)\leq O\big(T_f(r, \omega)\big).$$
This implies that 
 $$ \int_{\partial \Delta(r)}\log \xi d\pi_r \leq O\big(\log T_f(r, \omega)+\delta\log r\big). $$
By this with (\ref{gjj}), we prove the theorem.
\end{proof}

  \section{Carlson-Griffiths Theory for Parabolic Complete K\"ahler  Manifolds with Nonnegative Ricci Curvature}

In this section,  we use   $(M, g)$ to denote   a   parabolic complete noncompact  K\"ahler manifold  with nonnegative Ricci curvature,
  of complex dimension $m.$  Like before,   denote by $\nabla$    the gradient operator on  $M,$ and 
by $\alpha$     the   K\"ahler form of $M$ associated to $g$   defined by 
 $$\alpha=\frac{\sqrt{-1}}{\pi}\sum_{i,j=1}^mg_{i\bar j}dz_i\wedge d\bar z_{j}$$
in a local holomorphic coordinate system $z_1,\cdots,z_m.$
  
    \subsection{Construction of $\Delta(r)$}~\label{sec41}

Fix a reference point $o\in M.$  From   Theorem \ref{ppttt}, for any $0<\epsilon<1,$ there exist two constants $c_1,c_2>0$ such that   
\begin{equation}\label{eett}
\frac{c_1}{2V(\sqrt{t})}e^{-\frac{\rho(x)^2}{4(1-\epsilon)t}}\leq p(t,o,x)\leq \frac{c_2}{2V(\sqrt{t})}e^{-\frac{\rho(x)^2}{4(1+\epsilon)t}}
\end{equation}
holds for all $x\in M$ and all $t>0,$
in which $V(r)=V_o(r)$ and $\rho(x)=\rho(o,x).$

Define 
$$G_r(o,x)=2\int_0^rp(t,o,x)dt.$$
Using the properties of $p(t,o,x),$
we have for  $r>0$  
\begin{equation}\label{pro}
\lim_{\rho(x)\to 0} G_r(o,x)=\infty, \ \ \ \    \lim_{\rho(x)\to \infty} G_r(o,x)=0.
\end{equation}
For $r>0,$ define 
$$\Delta(r)=\left\{x\in M: \     G_r(o,x)>c_1\int_0^r\frac{1}{V(\sqrt{t})}e^{-\frac{r^2}{4(1-\epsilon)t}}dt\right\}.$$
With the help  of   (\ref{pro}),  we  deduce  that $\Delta(r)$ is a precompact domain containing $o$ in $M,$  such that  $\overline{\Delta(r_1)}\subseteq\Delta(r_2)$ whenever  $r_1<r_2$ and that 
$$\lim_{r\to 0}\Delta(r)=\emptyset \ \ \ \    \lim_{r\to \infty} \Delta(r)=M.$$ 
Hence,  the family $\{\Delta(r)\}_{r>0}$ exhausts $M.$ The boundary $\partial\Delta(r)$ of $\Delta(r)$ can be described as
$$\partial\Delta(r)=\left\{x\in M: \     G_r(o,x)=c_1\int_0^r\frac{1}{V(\sqrt{t})}e^{-\frac{r^2}{4(1-\epsilon)t}}dt\right\}.$$
 By Sard's theorem,  $\partial\Delta(r)$ is a smooth submanifold of $M$  for almost all $r>0.$

Set 
$$h_r(o,x)=G_r(o,x)-c_1\int_0^r\frac{1}{V(\sqrt{t})}e^{-\frac{r^2}{4(1-\epsilon)t}}dt.$$
Note that $h_r(o,x)$ has the following properties:   
     $$\begin{cases}
 -\frac{1}{2}\Delta h_r(o,x)=\delta_o(x)-p(r,o,x),  \ \ &    x\in\Delta(r); \\
 h_r(o,x)>0, \ \  &  x\in\Delta(r);  \\
  h_r(o,x)=0, \ \  &  x\in\partial\Delta(r),
\end{cases}$$
 where $\delta_o$ is the Dirac's delta  function with a pole at $o.$ 
It is  clear that  $h_r(o,x)$ is not the Green function for $\Delta(r).$
 Define the  positive  measure $\Theta_r$  on $\partial\Delta(r)$ with respect to $o$  by
   $$d\Theta_r=\frac{1}{2}\frac{\partial h_r(o,x)}{\partial{\vec{\nu}}}d\sigma_r,$$
     where  $\partial/\partial \vec\nu$ is the inward  normal derivative on $\partial \Delta(r),$ and $d\sigma_{r}$ is the induced Riemannian area element of 
$\partial \Delta(r).$ Note  that  $\Theta_r$ is not a harmonic measure on  $\partial\Delta(r),$ and also not  a probabilistic measure on  $\partial\Delta(r)$ due to 
$$\int_{\partial\Delta(r)}d\Theta=1-\int_{\Delta(r)}p(r,o,x)dv<1,$$
which can be   obtained  by Green's first or second identity (see Corollary \ref{dynkin2} later).

\begin{theorem}\label{beta}  For any  $0<\epsilon<1,$ we have 
$$r\leq\min_{x\in\partial\Delta(r)}\rho(x) \leq \max_{x\in\partial\Delta(r)} \rho(x)\leq  \beta r$$
holds for all sufficiently large   $r>0,$ where $\beta=(1+\epsilon)/(1-\epsilon).$  
\end{theorem}

\begin{proof}  Combining  the definition of $\Delta(r)$ with the lower estimate of $p(t,o,x)$ in  (\ref{eett}),  it is not hard  to deduce that $B(r)\subseteq \Delta(r)$ which implies  
$$r\leq\min_{x\in\partial\Delta(r)}\rho(x)$$
for all $r>0.$
Conversely,  assume that the last  inequality given  in the theorem is false.  There are sequences $\{x_n\}, \{r_n\}$ with $x_n\in\partial\Delta(r_n)$ 
and $0<r_n\to\infty$ as $n\to\infty,$ such that
\begin{equation}\label{ccoo}
 \rho(x_n)>  \frac{1+\epsilon}{1-\epsilon}r_n, \ \ \ \    n=1,2,\cdots
 \end{equation}
 Employing  the definition of $\partial\Delta(r)$ and the  upper estimate of $p(t,o,x)$ in  (\ref{eett}), we can arrive at  
$$\int_0^{r_n}\frac{1}{V(\sqrt{t})}e^{-\frac{r_{n}^2}{4(1-\epsilon)t}}dt\leq \frac{c_2}{c_1}\int_0^{r_n}\frac{1}{V(\sqrt{t})}e^{-\frac{\rho(x_n)^2}{4(1+\epsilon)t}}dt, \ \ \ \     n=1,2,\cdots$$
Applying    (\ref{ccoo}),   the difference between the  right hand side (RHS) and  the left hand side (LHS) of the above inequality satisfies that     for $n=1,2,\cdots$
   \begin{eqnarray*}
0\leq {\rm{RHS-LHS}} 
&<& \int_0^{r_n}\frac{1}{V(\sqrt{t})}\left(\frac{c_2}{c_1}e^{-\frac{(1+\epsilon)r_n^2}{4(1-\epsilon)^2t}}-e^{-\frac{r_n^2}{4(1-\epsilon)t}}\right)dt \\
&\leq& \int_0^{r_n}\frac{1}{V(\sqrt{t})}\left(\frac{c_2}{c_1}e^{-\frac{(1+\epsilon)r_n^2}{4(1-\epsilon)t}}-e^{-\frac{r_n^2}{4(1-\epsilon)t}}\right)dt. 
   \end{eqnarray*}
Write  
$$\frac{c_2}{c_1}e^{-\frac{(1+\epsilon)r_n^2}{4(1-\epsilon)t}}-e^{-\frac{r_n^2}{4(1-\epsilon)t}} 
= e^{-\frac{r_n^2}{4(1-\epsilon)t}}\left(e^{-{\frac{\epsilon r_n^2}{4(1-\epsilon)t}}+\log\frac{c_2}{c_1}}-1\right).$$
Since $t\leq r_n$ and $r_n\to\infty$ as $n\to\infty,$ when $n$ is sufficiently large 
$$-{\frac{\epsilon r_n^2}{4(1-\epsilon)t}}+\log\frac{c_2}{c_1} \leq -{\frac{\epsilon r_n}{4(1-\epsilon)}}+\log\frac{c_2}{c_1} 
<0,$$
which implies that
$$e^{-{\frac{\epsilon r_n^2}{4(1-\epsilon)t}}+\log\frac{c_2}{c_1}}-1<0.$$
Thus,   we deduce that  
$0\leq{\rm{RHS-LHS}}<0$ for $n$ large enough. 
but which is  a contradiction. This completes the proof. 
\end{proof}

\begin{cor}\label{domain}   For any  $0<\epsilon<1,$ we have  
$$B(r)\subseteq\Delta(r)\subseteq B(\beta r)$$
holds for all sufficiently large   $r>0,$ where $\beta=(1+\epsilon)/(1-\epsilon).$ 
\end{cor}

\subsection{Nevanlinna's Functions and First Main Theorem}~\label{sec42}

    We   introduce  Nevanlinna's functions in the parabolic setting. 
  Let $(N, \omega)$ be a compact K\"ahler manifold, and  $D$  an effective devisor  cohomologous to $\omega$ on $N,$ i.e.,  there exists a function 
$u_D\geq0$ on $N$ such that 
\begin{equation*}\label{}
\omega-[D]=2dd^c[u_D]
\end{equation*}
in the sense of currents.     Let $f: M\to N$ be a meromoprhic  mapping. Denote by $e_{f,\omega}$  the energy density function of $f$ with respect to metrics $\alpha, \omega,$ defined as in Section \ref{sec32}. 
Similarly, the \emph{characteristic function, proximity function, counting function} and \emph{simple counting function} of $f$ are  defined  respectively as  
   \begin{eqnarray*}
T_f(r, \omega)&=& \frac{1}{2}\int_{\Delta(r)}h_r(o,x)e_{f,\omega}dv, \\
m_f(r,D)&=&\int_{\partial\Delta(r)}u_D\circ fd\Theta_r, \\
N_f(r,D)&=& \frac{\pi^m}{(m-1)!}\int_{f^*D\cap\Delta(r)}h_r(o,x)\alpha^{m-1}, \\
\overline{N}_f(r, D)&=& \frac{\pi^m}{(m-1)!}\int_{f^{-1}(D)\cap\Delta(r)}h_r(o,x)\alpha^{m-1}. 
 \end{eqnarray*}
 
 \begin{remark} We make a comparison of  the  Nevanlinna's functions in these two settings (non-parabolic and parabolic settings). 
 Let $g_r(o,x)$ denote    the  Green function of $\Delta/2$ for $\Delta(r)$ with a pole at $o$ satisfying Dirichlet boundary condition. 
   Let  $\pi_r$ denote   the  harmonic measure on $\partial\Delta(r)$ with respect to $o.$
In the non-parabolic setting,   the characteristic function, proximity function, counting function and simple counting function of $f$ are  defined  respectively 
as  (to distinguish, we mark a star in the upper right corner of the notations)
   \begin{eqnarray*}
T^*_f(r, \omega)&=& \frac{1}{2}\int_{\Delta(r)}g_r(o,x)e_{f,\omega}dv, \\
m^*_f(r,D)&=&\int_{\partial\Delta(r)}u_D\circ fd\pi_r, \\
N^*_f(r,D)&=& \frac{\pi^m}{(m-1)!}\int_{f^*D\cap\Delta(r)}g_r(o,x)\alpha^{m-1}, \\
\overline{N}^*_f(r, D)&=& \frac{\pi^m}{(m-1)!}\int_{f^{-1}(D)\cap\Delta(r)}g_r(o,x)\alpha^{m-1}. 
 \end{eqnarray*}
 Since  
      $$\begin{cases}
 \Delta\big(h_r(o,x)-g_r(o,x)\big)=2p(r,o,x)>0,  \ \ &    x\in\Delta(r); \\
  h_r(o,x)-g_r(o,x)=0, \ \  &  x\in\partial\Delta(r),
\end{cases}$$
 we see that $h_r(o,x)-g_r(o,x)$ is  subharmonic  on $\Delta(r)$ and identically zero  on $\partial\Delta(r).$
 Using the maximum modulus principle, $h_r(o,x)-g_r(o,x)$ can attain its maximum only on $\partial\Delta(r).$ Thus , we deduce that  $h_r(o,x)\leq g_r(o,x).$
 From the definitions of $d\pi_r, d\Theta_r,$ we further obtain $d\pi_r\leq d\Theta_r.$ Hence, we have 
    \begin{eqnarray*}
T_f(r, \omega)&\leq& T^*_f(r, \omega), \ \ \ \  \ \ \     \ \  m_f(r,D)\leq m^*_f(r,D),    \\
N_f(r,D)&\leq& N^*_f(r,D), \ \ \ \  \ \  \  \  \overline{N}_f(r, D)\leq\overline{N}^*_f(r, D).
 \end{eqnarray*}
 \end{remark}
 
\begin{lemma}\label{dynkin1} Let $\phi$ be a $\mathscr C^2$ function on  $M$ outside a polar set of singularities at most. Assume that $\phi(o)\not=\infty.$  Then
$$\int_{\partial \Delta(r)}\phi d\Theta_{r}-\phi(o)=\frac{1}{2}\int_{\Delta(r)}h_r(o,x)\Delta \phi dv-\int_{\Delta(r)}p(r,o,x)\phi dv.$$
\end{lemma}
\begin{proof} 
The argument   is  similar to  the proof of Lemma \ref{dynkin}.
 We apply     Green's second identity to   obtain 
  \begin{eqnarray*}
&&\int_{\Delta(r)}h_r(o,x)\Delta \phi dv-\int_{\Delta(r)}\phi\Delta h_r(o,x) dv \\
&=& \int_{\partial\Delta(r)}\phi \frac{\partial h_r(o,x)}{\partial{\vec{\nu}}}d\sigma_r-\int_{\partial\Delta(r)}h_r(o,x) \frac{\partial\phi}{\partial{\vec{\nu}}}d\sigma_r,
  \end{eqnarray*}
  where  $\partial/\partial \vec\nu$ is the inward  normal derivative on $\partial \Delta(r),$ and $d\sigma_{r}$ is the induced Riemannian area element of 
$\partial \Delta(r).$  The properties   of  $h_r(o,x)$ give  that  
$$\int_{\Delta(r)}\phi\Delta h_r(o,x) dv=-2\phi(o)+2\int_{\Delta(r)}p(r,o,x)\phi dv$$
and 
$$\int_{\partial\Delta(r)}h_r(o,x) \frac{\partial\phi}{\partial{\vec{\nu}}}d\sigma_r=0.$$
Again, it yields from  the definition of $\Theta_r$ that  
$$ \int_{\partial\Delta(r)}\phi \frac{\partial h_r(o,x)}{\partial{\vec{\nu}}}d\sigma_r=2\int_{\partial\Delta(r)}\phi d\Theta_r.$$
Put together   the above, we can prove  the lemma.
\end{proof}

When $\phi=1,$    it is immediate  from Lemma \ref{dynkin1}   that   

\begin{cor}\label{dynkin2} We have 
$$\Theta_{r}\big(\partial \Delta(r)\big)=1-\int_{\Delta(r)}p(r,o,x) dv.$$
\end{cor}

The \emph{residual function} of $f$ with respect to $D$ is defined as 
$$E_f(r, D)=\int_{\Delta(r)}p(r,o,x)u_D\circ f dv.$$
We can see from  $\Delta(h_r(o,x)-g_r(o,x))=2p(r,o,x)$ that  the residual function   $E_f(r, D)$  
 arises  from the  residual between $h_r(o,x)$ and $g_r(o,x).$  In fact, using the property that 
  $h_r(o,x)-g_r(o,x)=0$ on $\partial\Delta(r)$ and  the maximum modulus principle, it is immediate that 
 $$h_r(o,x)-g_r(o,x)\not=0 \Longleftrightarrow \Delta\big(h_r(o,x)-g_r(o,x)\big)\not=0.$$
\ \ \  \   In what follows, we give   a probabilistic interpretation  of $E_f(r, D).$  Let $X_t$ be the  Brownian motion with generator $\Delta$ on $M,$ i.e., $X_t$ is a heat diffusion on $M,$ whose  
  transition density function  is  
the minimal positive fundamental solution of the heat equation
$$\Big(\Delta-\frac{\partial}{\partial t}\Big)u(t,x)=0.$$
Thus, in a probabilistic viewpoint, 
the  heat kernel $p(t,x,y)$  is   the  transition density function of $X_t,$ which describes  the probability density of $X_t$ started  at a point $x\in M$ to 
reach another point $y\in M$ after time $t.$
Let $\mathbb E$ 
denote    the expectation of $X_t$ started at $o.$ Then 
$$E_f(r,D)=\mathbb E\big[u_D\circ f(X_r):  X_r\in\Delta(r)\big].$$

    By Lemma \ref{dynkin1} and  the  arguments  in the proof of Theorem \ref{first}, it is trivial  to conclude    that 

\begin{theorem}[First Main Theorem]  Assume   that $f(o)\not\in{\rm{Supp}}D.$ Then
$$T_f(r, \omega)+u_D\circ f(o)=m_f(r,D)+N_f(r,D)+E_f(r,D).$$
\end{theorem}

\subsection{Calculus Lemma}~

To establish the Calculus Lemma, we first give an upper estimate of $d\Theta_r.$

\begin{lemma}\label{Theta} For any  $0<\epsilon<1,$ we have
$$d\Theta_r\leq \frac{c_2}{2}\int_0^r\frac{1}{V(\sqrt{t})}e^{-\frac{r^2}{4(1+\epsilon)t}}\frac{dt}{\sqrt{t}}\cdot d\sigma_r$$
holds on $\partial\Delta(t)$ with $0<t\leq r,$
where $c_2>0$ is  given by $(\ref{eett}).$
\end{lemma}

\begin{proof} By the definitions of $d\Theta_r$ and $h_r(o,x),$  we have 
   $$d\Theta_r= \frac{1}{2}\frac{\partial h_r(o,x)}{\partial\vec{\nu}}d\sigma_r
=  \frac{1}{2}\frac{\partial G_r(o,x)}{\partial\vec{\nu}}d\sigma_r=\int_0^r \frac{\partial p(t, o,x)}{\partial\vec{\nu}}dt\cdot d\sigma_r,$$
    where  $\partial/\partial \vec\nu$ is the inward  normal derivative on $\partial \Delta(r).$ 
   By Theorem \ref{gradient}
$$d\Theta_r\leq  \int_0^r \|\nabla p(t,o,x)\|dt\cdot d\sigma_r 
\leq  \frac{c_2}{2}\int_0^r\frac{1}{V(\sqrt t)}e^{-\frac{\rho(x)^2}{4(1+\epsilon)t}}\frac{dt}{\sqrt{t}}\cdot d\sigma_r.$$
  Since  $\rho(x)\geq r$ on $\partial\Delta(r)$ by Theorem \ref{beta},  we have  the desired result. 
\end{proof}

Set
$$F(r,\delta)=\frac{c_2(1-\epsilon)}{c_1}\frac{\displaystyle\int_0^r\frac{1}{V(\sqrt{t})}e^{-\frac{r^2}{4(1+\epsilon)t}}\frac{dt}{\sqrt{t}}}{\displaystyle\left(r\int_0^{r}\frac{1}{V(\sqrt{t})}e^{-\frac{r^2}{4(1-\epsilon)t}}\frac{dt}{t}\right)^{1+\delta}},$$
where $\epsilon, c_1,c_2$ are given by $(\ref{eett}).$

We now prove  the following Calculus Lemma: 
 \begin{theorem}[Calculus Lemma]\label{calculus11}
Let $k\geq0$ be a locally integrable function on $M.$ Assume that $k$ is locally bounded at $o.$ Then for any $\delta>0,$  there exists  a subset $E_{\delta}\subseteq(0,\infty)$ of finite Lebesgue measure such that
$$\int_{\partial\Delta(r)}kd\Theta_r\leq F(r,\delta)\bigg(\int_{\Delta(r)}h_{r}(o,x)kdv\bigg)^{(1+\delta)^2}$$
holds for all $r>0$ outside $E_{\delta}.$ 
\end{theorem}

\begin{proof} 
 Applying  the definition of $h_r(o,x)$ and the  lower estimate of $p(t,o,x)$ in  (\ref{eett}), we deduce that  
   \begin{eqnarray*}
&& \int_{\Delta(r)} h_{r}(o,x)kdv \\ 
&=&\int_0^{r}dt\int_{\partial \Delta(t)} h_{r}(o,x)kd\sigma_t \\
&\geq & c_1\int_0^{r}dt\int_0^{r}\frac{1}{V(\sqrt{s})}\bigg(e^{-\frac{t^2}{4(1-\epsilon)s}}-e^{-\frac{r^2}{4(1-\epsilon)s}}\bigg)ds\int_{\partial \Delta(t)}kd\sigma_t \\
&=& \int_0^{r}B(r,t)dt \int_{\partial \Delta(t)}kd\sigma_t,
  \end{eqnarray*}
  where  
$$B(r,t)=c_1\int_0^{r}\frac{1}{V(\sqrt{s})}\bigg(e^{-\frac{t^2}{4(1-\epsilon)s}}-e^{-\frac{r^2}{4(1-\epsilon)s}}\bigg)ds.$$
Set 
$$\Gamma(r)=\int_0^{r}B(r,t)dt \int_{\partial \Delta(t)}kd\sigma_t.$$
Since $B(r,r)=0$ and 
   \begin{eqnarray*}
&&\frac{\partial}{\partial r}B(r,t) \\
&=& \frac{c_1}{V(\sqrt{r})}\bigg(e^{-\frac{t^2}{4(1-\epsilon)r}}-e^{-\frac{r}{4(1-\epsilon)}}\bigg)
+\frac{c_1r}{2(1-\epsilon)}\int_0^{r}\frac{1}{V(\sqrt{s})}e^{-\frac{r^2}{4(1-\epsilon)s}}\frac{ds}{s} \\
&\geq& \frac{c_1r}{2(1-\epsilon)}\int_0^{r}\frac{1}{V(\sqrt{s})}e^{-\frac{r^2}{4(1-\epsilon)s}}\frac{ds}{s}, 
   \end{eqnarray*}
we obtain
   \begin{eqnarray*}
\Gamma'(r)&=& \int_0^{r}\frac{\partial}{\partial r}B(r,t) dt \int_{\partial \Delta(t)}kd\sigma_t \\
&\geq& \frac{c_1r}{2(1-\epsilon)}\int_0^{r}\frac{1}{V(\sqrt{s})}e^{-\frac{r^2}{4(1-\epsilon)s}}\frac{ds}{s}\int_0^rdt\int_{\partial \Delta(t)}kd\sigma_t.
   \end{eqnarray*}
Set
$$\phi(r)= r\int_0^{r}\frac{1}{V(\sqrt{s})}e^{-\frac{r^2}{4(1-\epsilon)s}}\frac{ds}{s}.$$
Then
$$\frac{d}{dr}\frac{\Gamma'(r)}{\phi(r)}=\frac{c_1}{2(1-\epsilon)}\int_{\partial \Delta(r)}kd\sigma_r.$$
Using Borel's growth lemma twice,  then for any $\delta>0,$ there exists a subset $E_\delta\subseteq(0,\infty)$ of finite Lebesgue measure such that 
   \begin{eqnarray*}
\int_{\partial \Delta(r)}kd\sigma_r &\leq& \frac{2(1-\epsilon)}{c_1}\phi(r)^{-(1+\delta)}\Gamma(r)^{(1+\delta)^2}      \\
&\leq& \frac{2(1-\epsilon)}{c_1}\left(r\int_0^{r}\frac{1}{V(\sqrt{s})}e^{-\frac{r^2}{4(1-\epsilon)s}}\frac{ds}{s}\right)^{-(1+\delta)}\Gamma(r)^{(1+\delta)^2}
   \end{eqnarray*}
   holds for all $r>0$ outside $E_\delta.$
On the other hand,  Lemma \ref{Theta} gives 
$$\int_{\partial \Delta(r)}kd\Theta_r\leq \frac{c_2}{2}\int_0^r\frac{1}{V(\sqrt{t})}e^{-\frac{r^2}{4(1+\epsilon)t}}\frac{dt}{\sqrt{t}} \int_{\partial \Delta(r)}kd\sigma_r.$$
To conclude, we have 
   \begin{eqnarray*}
\int_{\partial \Delta(r)}kd\Theta_r &\leq&  \frac{c_2(1-\epsilon)}{c_1}\frac{\displaystyle\int_0^r
\frac{1}{V(\sqrt{t})}e^{-\frac{r^2}{4(1+\epsilon)t}}\frac{dt}{\sqrt{t}}}{\displaystyle\left(r\int_0^{r}\frac{1}{V(\sqrt{t})}e^{-\frac{r^2}{4(1-\epsilon)t}}\frac{dt}{t}\right)^{1+\delta}}\Gamma(r)^{(1+\delta)^2} \\
&=& F(r,\delta)\Gamma(r)^{(1+\delta)^2} \\
&\leq& F(r,\delta)\bigg(\int_{\Delta(r)}h_{r}(o,x)kdv\bigg)^{(1+\delta)^2}
   \end{eqnarray*}
      holds for all $r>0$ outside $E_\delta.$
\end{proof}

\begin{lemma}  For any $0<\epsilon<1,$ we have 
$$\log^+F(r,\delta)\leq O(r).$$
\end{lemma}

\begin{proof}  Since $M$ has nonnegative Ricci curvature,  it yields  from   Theorem \ref{comp} and Theorem \ref{volume1} that 
$V(\sqrt{r})/V(\sqrt{t})\leq r^m/t^m$
for $0<t\leq r$ and $$O(r)\leq V(\sqrt{r})\leq O(r^m)$$
for $r>\epsilon_0,$
where $\epsilon>0$ is an arbitrarily small constant. 
For the numerator of $F(r,\delta),$  we have  
   \begin{eqnarray*} 
\int_0^r\frac{1}{V(\sqrt{t})}e^{-\frac{r^2}{4(1+\epsilon)t}}\frac{dt}{\sqrt{t}} &=& \frac{1}{V(\sqrt{r})} \int_0^r\frac{V(\sqrt{r})}{V(\sqrt{t})}e^{-\frac{r^2}{4(1+\epsilon)t}}\frac{dt}{\sqrt{t}} \\
&\leq&  \frac{r^m}{V(\sqrt{r})} \int_0^re^{-\frac{r^2}{4(1+\epsilon)t}}\frac{dt}{t^{m+\frac{1}{2}}} \\
&\leq& \frac{r^m}{V(\sqrt{r})}\bigg(\int_1^re^{-\frac{r^2}{4(1+\epsilon)t}}\frac{dt}{t^{m+\frac{1}{2}}}+O(1)\bigg) \\
&\leq& \frac{r^m}{V(\sqrt{r})}\big(r+O(1)\big).
   \end{eqnarray*} 
For the denominator of $F(r, \delta),$ we have  
   \begin{eqnarray*} 
\int_0^{r}\frac{1}{V(\sqrt{t})}e^{-\frac{r^2}{4(1-\epsilon)t}}\frac{dt}{t} &\geq& \int_{\frac{r}{2}}^r\frac{1}{V(\sqrt{t})}e^{-\frac{r^2}{4(1-\epsilon)t}}\frac{dt}{t} \\
&\geq&\frac{1}{V(\sqrt{r})} \int_{\frac{r}{2}}^re^{-\frac{r^2}{4(1-\epsilon)t}}\frac{dt}{t} \\
&\geq&\frac{r\log 2}{V(\sqrt{r})}e^{-\frac{r}{2(1-\epsilon)}}.
   \end{eqnarray*}
Put together the above, it is immediate that  
$$\log^+F(r,\delta)\leq O(r).$$
\end{proof}

It is therefore that 
 \begin{cor}\label{calculus12}
Let $k\geq0$ be a locally integrable function on $M.$ Assume that $k$ is locally bounded at $o.$ Then for any $\delta>0,$  there exists  a subset $E_{\delta}\subseteq(0,\infty)$ of finite Lebesgue measure such that
$$\log^+\int_{\partial\Delta(r)}kd\Theta_r\leq (1+\delta)^2\log^+\int_{\Delta(r)}h_{r}(o,x)kdv+O(r)$$
holds for all $r>0$ outside $E_{\delta}.$  
 \end{cor}

\subsection{Second Main Theorem}~

Below, we  prove  a Second Main Theorem in the parabolic case. 
 Let  $(N, \omega)$  be a compact  K\"ahler manifold of complex dimension not greater than that  of $M.$ 
   Let $D_1,\cdots, D_q$ be  effective divisors  on $N,$ which are  cohomologous to $\omega,$ i.e., there exist   functions $u_{D_1},\cdots, u_{D_q}\geq0$ on $N$ such that 
$$\omega-[D_j]=2dd^c[u_{D_j}], \ \ \ \  j=1,\cdots,q.$$ 
Set
$$\Psi=\frac{\omega^n}{\prod_{j=1}^q u_{D_j}^{2}e^{-2u_{D_j}}}.$$
Let $f:M\rightarrow N$ be a differentiably non-degenerate meromorphic mapping.   Define a function $\xi\geq0$ by    
$$f^*\Psi\wedge\alpha^{m-n}=\xi\alpha^m.$$
Applying  the similar arguments in  the proof of (\ref{gjj}) in Theorem \ref{main},  we can obtain 
 \begin{eqnarray}\label{gjj22}
&&  qT_f(r,\omega)+T_f(r, K_N)+T(r,\mathscr R)-\sum_{j=1}^q\overline{N}_f(r, D)  \\
&\leq& q\log^+T_f(r,\omega)+ \frac{1}{2}\int_{\partial \Delta(r)}\log \xi d\Theta_r+\frac{1}{2}\int_{\Delta(r)}p(r,o,x)\log \xi  dv \ \big\|. \nonumber
 \end{eqnarray}
 \ \ \ \   In order to give an upper estimate of 
   \begin{equation}\label{wpp}
\int_{\partial \Delta(r)}\log \xi d\Theta_r+\int_{\Delta(r)}p(r,o,x)\log \xi  dv
\end{equation} 
and thus establish the   Second Main Theorem, we  need some lemmas.

\begin{lemma}\label{mei1} For  any $0<\epsilon<1,$ we have 
$$h_{\theta r}(o,x)\geq \frac{c_1}{2}\Big(1-e^{-\frac{3}{16(1-\epsilon)}}\Big)\frac{r}{V(\sqrt{r})}e^{-\frac{\theta^2 r}{8(1-\epsilon)}}$$
holds on $\Delta(r)$ for all sufficiently large $r>0,$ where $\theta=2(1+\epsilon)/(1-\epsilon)$ and $c_1$ is given by in $(\ref{eett}).$ 
\end{lemma}

\begin{proof}  By  Theorem \ref{beta},  we obtain    $\rho(x)\leq\theta r/2$ on $\Delta(r)$ for sufficiently large $r>0.$ 
 Applying the  lower estimate of $p(s,o,x)$ in (\ref{eett}),  it yields  that   
   \begin{eqnarray*} 
h_{\theta r}(o,x)&=& 2\int_0^{\theta r}p(s, o,x)ds-c_1\int_0^{\theta r}\frac{1}{V(\sqrt{s})}e^{-\frac{\theta^2 r^2}{4(1-\epsilon)s}}ds \\
&\geq& c_1\int_0^{\theta r}\frac{1}{V(\sqrt{s})}e^{-\frac{\rho(x)^2}{4(1-\epsilon)s}}ds-c_1\int_0^{\theta r}\frac{1}{V(\sqrt{s})}e^{-\frac{\theta^2 r^2}{4(1-\epsilon)s}}ds \\
&\geq&  \frac{c_1}{V(\sqrt{r})}\int_0^{r}\bigg(e^{-\frac{\theta^2 r^2}{16(1-\epsilon)s}}-e^{-\frac{\theta^2 r^2}{4(1-\epsilon)s}}\bigg)ds \\
&\geq&  \frac{c_1}{V(\sqrt{r})}\int_0^{r}e^{-\frac{\theta^2 r^2}{16(1-\epsilon)s}}\Big(1-e^{-\frac{3r}{16(1-\epsilon)}}\Big)ds  \\
   \end{eqnarray*} 
on  $\Delta(r).$    Hence, for   sufficiently large  $r>1$
      \begin{eqnarray*} 
   h_{\theta r}(o,x)&\geq& \frac{c_1}{V(\sqrt{r})}\int_0^{r}e^{-\frac{\theta^2 r^2}{16(1-\epsilon)s}}\Big(1-e^{-\frac{3}{16(1-\epsilon)}}\Big)ds \\
 &=& \Big(1-e^{-\frac{3}{16(1-\epsilon)}}\Big)\frac{c_1}{V(\sqrt{r})}\int_0^{r}e^{-\frac{\theta^2 r^2}{16(1-\epsilon)s}}ds \\
&\geq& \Big(1-e^{-\frac{3}{16(1-\epsilon)}}\Big)\frac{c_1}{V(\sqrt{r})}\int_{\frac{r}{2}}^{r}e^{-\frac{\theta^2 r^2}{16(1-\epsilon)s}}ds \\
&\geq& \Big(1-e^{-\frac{3}{16(1-\epsilon)}}\Big)\frac{c_1}{V(\sqrt{r})}\int_{\frac{r}{2}}^{r}e^{-\frac{\theta^2 r}{8(1-\epsilon)}}ds \\
&=& \Big(1-e^{-\frac{3}{16(1-\epsilon)}}\Big)\frac{c_1r}{2V(\sqrt{r})}e^{-\frac{\theta^2 r}{8(1-\epsilon)}}
   \end{eqnarray*} 
   on $\Delta(r).$ 
\end{proof}

\begin{cor}\label{mei2}  For  any $0<\epsilon<1,$ we have
$$\log^+\int_{\Delta(r)}\xi^{\frac{1}{n}}dv\leq \log^+\int_{\Delta(\theta r)}h_{\theta r}(o,x)\xi^{\frac{1}{n}}dv+O(r),$$
 where $\theta=2(1+\epsilon)/(1-\epsilon).$  
\end{cor}

\begin{proof}  When $r>0$ is large enough,   it yields  from Lemma \ref{mei1} that 
   \begin{eqnarray*} 
   \int_{\Delta(r)}\xi^{\frac{1}{n}}dv&=& \int_{\Delta(r)}\frac{h_{\theta r}(o,x)}{h_{\theta r}(o,x)}\xi^{\frac{1}{n}}dv \\
   &\leq& \frac{2}{c_1}\Big(1-e^{-\frac{3}{16(1-\epsilon)}}\Big)^{-1}\frac{V(\sqrt{r})}{r}e^{\frac{\theta^2 r}{8(1-\epsilon)}}\int_{\Delta(\theta r)}h_{\theta r}(o,x)\xi^{\frac{1}{n}}dv. 
   \end{eqnarray*} 
It is therefore 
$$\log^+\int_{\Delta(r)}\xi^{\frac{1}{n}}dv\leq \log^+\int_{\Delta(\theta r)}h_{\theta r}(o,x)\xi^{\frac{1}{n}}dv+O(r).$$
\end{proof}

\begin{lemma}\label{wppp}  For  any $0<\epsilon<1,$ we have
$$\int_{\Delta(r)}p(r,o,x)\log\xi^{\frac{1}{n}}dv\leq \frac{c_2\theta^{2m}r^m}{2^{2m+1}}\bigg(\log^+\int_{\Delta(\theta r)}h_{\theta r}(o,x)\xi^{\frac{1}{n}}dv +O(r)\bigg),$$
 where $\theta=2(1+\epsilon)/(1-\epsilon)$ and $c_2$ is given  by  $(\ref{eett}).$
\end{lemma}

\begin{proof}  Note that   $\rho(x)\geq t$ on $\partial\Delta(t)$ since Theorem \ref{beta}.  Applying   the upper estimate of $p(r, o, x)$ in  (\ref{eett}), we arrive at   
$$p(r,o,x)\leq \frac{c_2}{2V(\sqrt{r})}e^{-\frac{t^2}{4(1+\epsilon)r}}$$
on $\partial\Delta(t).$
Invoking  Corollary \ref{domain} and Jensen's inequality,  we are led to 
   \begin{eqnarray*} 
 \int_{\Delta(r)}p(r,o,x)\log\xi^{\frac{1}{n}}dv 
&=& \int_0^rdt\int_{\partial\Delta(t)}p(r,o,x)\log\xi^{\frac{1}{n}}d\sigma_t \\
&\leq&  \frac{c_2}{2V(\sqrt{r})} \int_0^re^{-\frac{t^2}{4(1+\epsilon)r}}dt\int_{\partial\Delta(t)}\log\xi^{\frac{1}{n}}d\sigma_t \\
&\leq&  \frac{c_2}{2V(\sqrt{r})} \int_0^rdt\int_{\partial\Delta(t)}\log\xi^{\frac{1}{n}}d\sigma_t \\
&=&  \frac{c_2}{2V(\sqrt{r})} \int_{\Delta(r)}\log\xi^{\frac{1}{n}}dv \\
&\leq&  \frac{c_2{\rm{Vol}}(\Delta(r))}{2V(\sqrt{r})} \log\int_{\Delta(r)}\xi^{\frac{1}{n}}\frac{dv}{{\rm{Vol}}(\Delta(r))} \\
&\leq&  \frac{c_2V(\beta r)}{2V(\sqrt{r})} \log\int_{\Delta(r)}\xi^{\frac{1}{n}}\frac{dv}{V(r)} \\
&\leq&  \frac{c_2V(\beta r)}{2V(\sqrt{r})}\bigg(\log^+\int_{\Delta(r)}\xi^{\frac{1}{n}}dv +O(1)\bigg). 
   \end{eqnarray*} 
Since $M$ has nonnegative Ricci curvature,    Theorem \ref{comp} gives  
$$\frac{V(\beta r)}{V(\sqrt{r})}\leq \frac{\beta^{2m}r^{2m}}{r^m}=\beta^{2m}r^m=\frac{\theta^{2m}r^m}{2^{2m}}.$$
It is therefore 
$$\int_{\Delta(r)}p(r,o,x)\log\xi^{\frac{1}{n}}dv\leq \frac{c_2\theta^{2m}r^m}{2^{2m+1}}\bigg(\log^+\int_{\Delta(r)}\xi^{\frac{1}{n}}dv +O(1)\bigg).$$
Using  Corollary \ref{mei2} again, then we have the lemma proved.
\end{proof}

\begin{theorem}[Second Main Theorem]\label{main220}   Let $M$ be a parabolic complete noncompact   K\"ahler manifold with nonnegative Ricci curvature. 
 Let $(N, \omega)$ be a compact  K\"ahler manifold of complex dimension not greater than that  of $M.$
 Let $D_1,\cdots, D_q$ be effective divisors in general position on $N$ such that each $D_j$ is cohomologous to $\omega.$
  Let $f:M\rightarrow N$ be a differentiably non-degenerate meromorphic mapping.  Assume that $q\omega-{\rm Ric}(\omega^n)>0.$ Then for any $0<\epsilon<1,$ we have 
  $$qT_f(r, \omega)+T_f(r, K_N)+T(r, \mathscr R)\leq \sum_{j=1}^q\overline N_f(r,D_j)+O\left(r^m\log^+T_f(\theta r,\omega)\right) \big\|$$
holds,  where $\theta=2(1+\epsilon)/(1-\epsilon).$  
\end{theorem}

\begin{proof} 
As noted that we need to give an  upper estimate of (\ref{wpp}). 
By Corollary \ref{dynkin2} and Jensen's inequality 
 \begin{eqnarray*}
\int_{\partial \Delta(r)}\log \xi d\Theta_r &\leq& n\Theta_r(\partial\Delta(r))\log\int_{\partial \Delta(r)}\xi^{\frac{1}{n}} \frac{d\Theta_r}{ \Theta_r(\partial\Delta(r))} \\
&\leq& n\log\int_{\partial \Delta(r)}\xi^{\frac{1}{n}} d\Theta_r+n\Theta_r(\partial\Delta(r))\log\frac{1}{\Theta_r(\partial\Delta(r))} \\
&\leq& n\log^+\int_{\partial \Delta(r)}\xi^{\frac{1}{n}} d\Theta_r+O(1).
 \end{eqnarray*}
Again, combined with Corollary \ref{calculus12} and (\ref{oppo}), we deduce that 
 \begin{eqnarray*}
 \int_{\partial \Delta(r)}\log \xi d\Theta_r 
&\leq&  n\log^+\int_{\partial \Delta(r)}\xi^{\frac{1}{n}} d\Theta_r+O(1) \\
&\leq& n(1+\delta)^2\log^+ \int_{\Delta(r)}h_{r}(o,x)\xi^{\frac{1}{n}}dv+O(r) \\
&\leq& n(1+\delta)^2\log^+ \int_{\Delta(r)}h_{r}(o,x)\frac{-f^*{\rm Ric\Psi}\wedge\alpha^{m-1}}{\alpha^m}dv+O(r) \\
&\leq& n(1+\delta)^2\log^+ T_f(r, -{\rm Ric}\Psi)+O(r) \\
&\leq & n(1+\delta)^2\log^+ T_f(r, \omega)+O(r) \ \big\|.
 \end{eqnarray*}
 For the second term in (\ref{wpp}), we derive  from Lemma \ref{wppp} that 
  \begin{eqnarray*}
  \int_{\Delta(r)}p(r,o,x)\log\xi dv 
&=& n\int_{\Delta(r)}p(r,o,x)\log\xi^{\frac{1}{n}}dv \\ 
&\leq& \frac{nc_2\theta^{2m}r^m}{2^{2m+1}}\bigg(\log^+\int_{\Delta(\theta r)}h_{\theta r}(o,x)\xi^{\frac{1}{n}}dv +O(r)\bigg).
 \end{eqnarray*}
To conclude, we obtain  
$$ \int_{\partial \Delta(r)}\log \xi d\Theta_r+\int_{\Delta(r)}p(r,o,x)\log \xi  dv\leq O\big(r^m\log^+T_f(\theta r,\omega)\big) \ \big\|.$$
By this with (\ref{gjj22}), we finally completes the proof. 
\end{proof}



\end{document}